\documentclass[12pt]{amsart}
\usepackage{amssymb}

\newtheorem{theorem}{Theorem}[section]

\newtheorem{coro}[theorem]{Corollary}
\newtheorem{lemma}[theorem]{Lemma}
\newtheorem{definition}[theorem]{Definition}
\newtheorem{ex}[theorem]{Example}
\newtheorem{assum}{Assumption}[section]

\renewcommand{\epsilon}{\varepsilon}

\newcommand{\cdo}{\, \cdot \,}

\newcommand{\ro}{\rho}

\newcommand\C{\mathbf{C}}

\newcommand\R{\mathbf{R}}

\newcommand\Z{\mathbf{Z}}
\newcommand\wet{\widetilde{\eta}}

\newcommand{\supp}{\operatorname{supp}}

\renewcommand{\theenumi}{\alph{enumi}}
\renewcommand{\labelenumi}{\textrm{(\theenumi)}}

\begin{document}

\setcounter{section}{0}
\renewcommand{\theenumi}{\alph{enumi}}
\renewcommand{\labelenumi}{\textrm{(\theenumi)}}

\date{28/07/2008}

\title[Multivariable spectral multipliers and  quasielliptic operators]{Multivariable spectral multipliers and analysis of quasielliptic operators on fractals}

\author{Adam Sikora}

\address{Adam Sikora, Mathematical Sciences Institute,
Bldg 27,
Australian National University,
ACT 0200 Australia
}

\subjclass[2000]{Primary 42B15; Secondary 43A85, 28A80 }

\keywords{Spectral multipliers, analysis on fractals}

\begin{abstract}
We study multivariable spectral   multipliers $F(L_1,L_2)$ acting on Cartesian product of ambient spaces of two self-adjoint operators $L_1$ and  $L_2$. We prove that
if $F$ satisfies H\"ormander type differentiability condition then the operator
$F(L_1,L_2)$ is of Calder\'on-Zygmund type. We apply  obtained results to the analysis of
quasielliptic operators acting on product of some fractal spaces.  The existence and surprising properties of 
 quasielliptic operators have been recently observed in works of Bockelman, Drenning  and Strichartz. 
  This paper demonstrates that Riesz type operators corresponding to quasielliptic operators are continuous on $L^p$ spaces. This solves the problem posed in \cite[(1.3) p. 1363]{BS}. 
 \end{abstract}

\maketitle{}

I dedicate this paper to the memory of my teachers Andrzej Hulanicki and Tadeusz Pytlik. 

\section{Introduction}

Suppose that $L$ is  a  self-adjoint operator acting
on $L^2(X,\mu)$, where $X$ is a measure space with measure $\mu$.  Such an
operator admits a spectral  resolution  $E_L(\lambda)$  and for  any
bounded Borel function $F\colon \R \to \C$, we define the operator $F(L)$
by the formula
  \begin{equation*}
  F(L)=\int_\R F(\lambda) d E_L(\lambda).
  \end{equation*}
By the spectral theorem the operator $F(L)$ is continuous on $L^2(X,\mu)$.
Spectral  multiplier  theorems  investigate  sufficient conditions on
function $F$ which  ensure  that the   operator  $F(L)$ extends  to a
bounded operator on  $L^q$ for some $q$ in the range  $1\le q \le \infty$.
The theory of spectral multipliers constitutes an important field of
Harmonic analysis and there exists a vast literature devoted to the topic,
see for example \cite{Ale1, Ale2, ChS, CoS, DM, DOS, He1, Her1, Mi, MM, MRS}
and references within.  The main aim of this paper is to develop a theory
of multivariable spectral multipliers of two  self-adjoint independent operators
acting on the Cartesian product of their ambient spaces. One could also consider here three or more 
operators but, for simplicity, we limit the discussion to the two dimensional version.

\subsection{Multivariable  spectral multiplier}
We begin our discussion by explaining the definition of multivariable 
spectral multipliers. We consider
two self-adjoint operators $L_j$, $j=1,2$ acting on spaces $L^2(X_j)$. The
tensor
product operators    $L_{1} \otimes 1$ and $ 1 \otimes L_{2}$ act on
$L^2({X}_1 \times
{X}_2)$, where ${X}_1 \times
{X}_2$ is the Cartesian product of $X_1$ and $X_2$ with the product
measure $\mu=\mu_1 \times \mu_2$.
To simplify notation we will write  $L_1$ and $L_2$ instead of $L_{1}
\otimes 1$ and $ 1 \otimes L_{2}$.
Note that there is a unique spectral decomposition $E$ such that for all
Borel subsets $A\subset \R^2$, $E(A)$ is a projection on $L^2({X}_1
\times{X}_2)$ and such that for any Borel subsets $A_j \subset \R$, $j=1,2$
one has
$$
E(A_1\times{A}_2)=E_{L_1}(A_1)\otimes E_{L_2}(A_2).
$$
Hence for any function $F\colon \R^2   \to \C$ one can define the operator
$F(L_1,L_2)$ acting as operators on space $L^2({X}_1 \times {X}_2,)$ by
the formula
  \begin{equation}\label{equwd}
  F(L_1,L_2)=\int_{\R^2}F(\lambda_1,\lambda_2) d E(\lambda_1,\lambda_2).
  \end{equation}
A straightforward variation of classical spectral theory arguments shows
that for any bounded Borel function
$F\colon \R^2   \to \C$ the operator $F(L_1,L_2)$ is continuous on
$L^2({X}_1 \times {X}_2)$ and its norm is bounded by $\|F\|_{L^\infty}$.
In this paper we are looking for necessary smoothness conditions on
function $F$ so that  the operator $F(L_1,L_2)$ is bounded on a range of
other $L^p({X}_1 \times
{X}_2)$ spaces.
The  condition on function $F$ which we use is a variant of the
differentiability condition in H\"ormdander-Mikhlin Fourier multiplier result, see \cite{Her1, Mi} and
\cite[Theorem~7.9.6]{Her2}.
 
One  motivation for multivariable spectral multiplier results comes from Riesz
transform like operators, for example
$$
\frac{aL_1+bL_2}{cL_1+dL_2} \quad \mbox{or} \quad \frac{L_1L_2}{(cL_1+dL_2)^2},
$$
where $a,b\in \R$ are arbitrary real numbers and $c,d>0$. Here we assume that the operators $L_i$ are positive for 
$i=1,2$. It is easy to note that then such operators are bounded on 
$L^2({X}_1 \times
{X}_2)$. We prove that under standard
assumptions such operators are of Calder\'on-Zygmund type, that is, they are bounded on all $L^p({X}_1 \times
{X}_2)$ spaces for 
$1<p<\infty$ and of weak type $(1,1)$. Of course one can consider a much larger  family of operators of
this type.
For the Laplace operators acting on some fractal spaces and for some real numbers $c,d$  the above 
operators are still bounded on $L^2({X}_1 \times
{X}_2)$ (and as we prove it in this paper on other $L^p({X}_1 \times
{X}_2)$ spaces $1<p< \infty$) {\em even though $cd <0$}. We describe this intriguing  phenomenon in the following 
section. 

\subsection{Product of fractals and quasielliptic operators}

In \cite{BS} and \cite{DS} it is shown that for the Laplace operators defined on some fractal spaces, the set of ratios of eigenvalues have gaps. 
This means that there are intervals $(\alpha,\beta)$, $0<\alpha<\beta$, such that for any two eigenvalues $\lambda_i$,
$\lambda_j$ of the Laplacian acting on the same fractal spaces one has $\frac{\lambda_j}{\lambda_i}\notin (\alpha,\beta)$.
 For example if one considers  the Dirichlet or Neumann 
Laplacian on Sierpi\'nski Gasket $SG$  then 
$\frac{\lambda_j}{\lambda_i}\notin (\alpha,\beta)$, where
\begin{equation}\label{gap}
\alpha=\lim_{n\to \infty} \frac{\psi_n(5)}{\psi_n(3)}\approx 2.0611 \quad \mbox{and} \quad 
 \beta=\lim_{n\to \infty} \frac{\psi_n(3)}{\psi_{n+1}(3)}\approx 2.4288
\end{equation}
and   $\psi (x) = (5 -\sqrt{25 - 4x})/2$,
see \cite{BS}. 
The existence of gaps in the set of ratios of eigenvalues has a  surprising consequence. Namely one can consider
the product of two copies of such fractal spaces and operators $L_1$ and $L_2$ which are copies of the same Laplace operator acting on first and second variable respectively. Now it is easy  to notice that if $(\alpha,\beta)$ is a gap in the set of ratios of eigenvalues $\alpha<\gamma < \beta$, $a,b\in \R$, $c,d>0$ and $\frac{d}{c}=\gamma$ then the operators
\begin{equation}\label{qell}
\frac{aL_1+bL_2}{cL_1-dL_2}  \quad \mbox{or} \quad \frac{L_1L_2}{(cL_1-dL_2)^2},
\end{equation}
are bounded on $L^2$. Following \cite{BS} we call operators of the form ${cL_1-dL_2} $ {\em quasi\-elliptic}. 
It was asked in \cite{BS} whether the above operators are bounded  on other $L^p$ spaces. This question is the initial  motivation of this paper. We prove that these operators are indeed bounded on all $L^p$ spaces for $1<p<\infty$ and of
weak type $(1,1)$.

 The idea of multivariable  spectral multipliers is of independent interest. It seems to be possible to obtain more general versions of  multivariable  spectral multipliers and some Marcinkiewicz type variations of these results. However, here we concentrate on obtaining a possibly simple proof of weak type $(1,1)$ for Riesz transform type operators corresponding to quasielliptic operators. 

\subsection{Doubling condition}
Before we state our main result we have to describe our basic assumptions.
As it is usually the case in theory of spectral multipliers we require
the doubling  condition and some version of
Gaussian estimates for semigroups generated by operators $L_j$, $j=1,2$,
see \cite{Ale2, DOS}.
We assume that the considered ambient spaces
$X_j$ $j=1,2$ are equipped with a Borel measure $\mu_j$ and  distances
  $\rho_j$. Let $B(x,r)= \{ y \in X \colon \, \ro (x,y) <  r \}$ be the
open ball  centred at $x$ and radius $r$.
  We suppose  throughout  that  ${X_1},X_2$ satisfy the  doubling condition, that
is  there  exists  a constant $C$ such that
    \begin{equation}\label{doubling1.1}
    \mu_j(B(x_j, 2r)) \le C  \mu_j(B(x_j, r)) \quad \forall x_j \in X_j,
\forall r > 0, j=1,2.
    \end{equation}
  Note that   (\ref{doubling1.1}) implies that there   exist  positive
constants $C$ and $d_1,d_2$ such that
    \begin{equation}\label{d}
    \mu_j(B(x_j, t r)) \le C (1+t)^{d_j} \mu_j(B(x_j, r))  \quad \forall
t > 0,
    x_j \in X_j, r > 0.
    \end{equation}
  In the sequel we always assume that (\ref{d}) holds. Note that all fractal spaces which we discuss here 
  satisfy condition (\ref{d}). In fact for these spaces $\mu(B(x, r)) \sim r^d$ for all $r\le 1$, see \cite{Ki, Str2}.  

\subsection{Kernels of operators}
Suppose that
$T$ is a bounded operator on $L^2(X).$  We say that a measurable function
$K_T\colon X^2 \to \C$ is  the (singular) kernel  of $T$  if
  \begin{equation}\label{abcde}
 \langle T f_1,f_2\rangle  =
 \int_X T f_1\overline{f_2} d\mu  =
 \int_X K_{T}(x,y) f_1(y)\overline{f_2(x)} d\mu(x) d\mu(y).
 \end{equation}
for all $f_1,f_2\in C_c(X)$ (for all $f_1,f_2\in C_c(X)$ such that
$\supp f_1 \cap \supp f_2 = \emptyset$ respectively).
It is well known that
if $T$ is bounded from $L^1(X)$ to~$L^q(X)$, where $1<q$, then
$T$ is a kernel operator, and
$$
\|T\|_{L^1 \to L^q}
= \sup_{y\in X} \| K_{T}(\cdot , y)\|_{L^{q}}\, .
$$
  In addition, if $S$ is continuous on $L^q(X)$ then for almost all $y\in X$ 
\begin{equation}\label{ker1}
K_{ST,y}=SK_{T,y}
\end{equation}
where $K_{T,y}(x)=K_T(x,y)$. 
Next we denote the weak   type $(1,1)$ norm  of  an operator $T$ on
a measure space  $(X,\mu)$ by $\|T\|_{L^1 \to L^{1,\infty}}
=\sup \lambda  \ \mu  ( \{x\in  X : \ |Tf(x)| >  \lambda \})$, where the
supremum is taken over $\lambda > 0$ and functions $f$ with $L^1$ norm
less than one; this is often called the ``operator norm'', though in fact
it is not a norm.

In the sequel we will always require the following Gaussian estimates
for the heat kernel corresponding to the operators $L_1$ and $L_2$.
  \begin{assum}\label{1b}
  Let $L_j$, $j=1,2$ be  self-adjoint positive definite operators. We
assume that
 the semigroups generated by $L_j$ on $L^2(X_j)$ have the  kernels
  $p^{[j]}_{t}(x_j,y_j)=K_{\exp(-tL_j)}(x_j,y_j)$ defined by  {\rm
(\ref{abcde})} which
  for some constants $C_j,b_j>0$ and ${m}>1$ satisfy the 
  following Gaussian upper bounds
    \begin{equation} \label{1.3kernel}
    \vert  p^{[j]}_{t}(x_j,y_j)  \vert \le C_j\mu(B(y_j, t^{1/m}))^{-1} \exp
    \Big(-b_j\frac{\ro_j(x_j,y_j)^{m/(m-1)}}{t^{1/(m-1)}} \Big)
    \end{equation}
  for all $t>0$ and $j\in \{1,2\}$.
  \end{assum}
  We  will call $p^{[j]}_{t}(x,y)$ the heat kernels  associated with $L_j$.
Such estimates are typical for elliptic or  sub-elliptic
differential operators  of order $m$ (see e.g. \cite{Da2}).
But such estimates hold also for most of the Laplace type operators acting
on fractals, see \cite{Ki, Str2}.

Note that Assumption~\ref{1b} implies that the operator
$K_{\exp(sL_1+tL_2)}$ has
a $L^\infty$ kernel given by the following formula
  $$
K_{\exp(sL_1+tL_2)}((x_1,x_2),(y_1,y_2))=p^{[1]}_{s}(x_1,y_1)p^{[2]}_t(x_2,y_2)
  $$
for all $s,t>0$. 
\section{Main result}
We define a family of dilations $\{\delta_t\}_{t>0}$ acting on functions $F\colon \R^2 \to C$  by
the formula
$$
\delta_{t}F(\lambda_1,\lambda_2)  =   F(t\lambda_1,t\lambda_2).
$$
Next let us recall that the norms in Sobolev spaces 
$W_{s}^{p}(\R^n)$ are defined by the formula
$$
\|F\|_{W_{s}^{p}(\R^n)}=\|(\Delta+I)^{s/2}F\|_{L^p(\R^n)},
$$ where $\Delta$ is the standard Laplace operator.
Now we can formulate our main spectral multiplier result.
  \begin{theorem}\label{main}
Suppose that  spaces $X_j$ satisfy doubling condition~{\rm (\ref{d})}
with constants $d_1$ and $d_2$. Further assume
that operators $L_j$ satisfy Assumption~{\rm \ref{1b}}  with  the same order $m$.
Let $F\colon \R^2 \to \C$ be a continuous function and let $\eta \in
C_c^\infty(0,\infty)$ be an auxiliary nonzero cut-off function.
Suppose that for some $s>\frac{d_1+d_2}{2}$
\begin{equation}\label{hc}
\sup_{t> 0} \Vert \widetilde{\eta} \, \delta_{t }F
\Vert_{W^\infty_s}<\infty,
\end{equation}
where $\widetilde{\eta}(x,y)=\eta(x+y)$.
Then  the operator $F(L_1,L_2)$ is of weak type $(1,1)$ and is bounded on
$L^q(X)$, $X=X_1\times X_2$
  for all $1<q<\infty$. In addition
    \begin{equation*}
    \|F({L_1,L_2})\|_{L^1 \to  L^{1,\infty}}
    \le  C_s \sup_{t >   0} \Vert \widetilde{\eta}
    \, \delta_t F \Vert_{W^\infty_s(\R^2)}.
    \end{equation*}
  \end{theorem}

{\em Remarks} 1. We assume that $L_1$ and $L_2$ are positive so 
$F(L_1,L_2)$ depends only on the restriction of  $F$ to
$[0,\infty)^2$. However, it is easier to state Theorem~\ref{main} if one
considers functions $F\colon \R^2 \to \C$.

2.  The condition on function $F$ in Theorem~\ref{main} is similar to the
condition in H\"ormdander-Mikhlin  Fourier multiplier result, see \cite{Her1, Mi} and
\cite[Theorem~7.9.6]{Her2}. The
only difference is that we use
the space $W^\infty_s$ instead of $W^2_s$. The significance of this difference
is discussed in \cite{DOS}.

3. It is not difficult to check that if for all $|I| \le \left[\frac{d_1+d_2}{2}\right]+1$,
where $[a]$ is an integer part of $a$, 
$$
\sup_{\lambda\in \R^2}|\lambda|^{|I|}|\partial^I F(\lambda)| < \infty,
$$
then $F$ satisfies condition (\ref{hc}), see \cite{Ale1, Ale2, LM}. In a sense condition (\ref{hc}) is a fractional exponent version of
the above condition in which $|I|$ must be an integer. The above condition is an illuminating  illustration of 
condition (\ref{hc}). In fact condition (\ref{hc}) can be stated  equivalently using interpolation between integer cases 
of the above definition. 

 Note that $\eta$    is only an auxiliary   function and that condition (\ref{hc}) does not  depend of  $\eta$.          
 
 4. It would be interesting to obtain a version of Theorem~\ref{main} with condition (\ref{hc}) replaced by 
$$
\sup_{\lambda\in \R^2}|\lambda^I\partial^I F(\lambda)| < \infty,
$$
for all $|I|\le l$ for some sufficiently large $l$, see \cite[Theorem 1.3]{MRS}.

5. It could be also interesting to try to  obtain some multivariable spectral multipliers results similar to Theorem~\ref{main}
using the techniques developed in \cite{DR, LLM}.

\subsection{Notation}
In subsequent sections we use the following notation. We put $X_P=X_1\times X_2$,
$x_P=(x_1,x_2)\in X_P$ and
$y_P=(y_1,y_2)$. Next $\mu_P=\mu_1\times \mu_2$ and we set $\rho_P(x_P,y_P)
=\max\{(\rho_1(x_1,y_1),\rho_2(x_2,y_2)\}$.  If  the discussed results hold
separately for both $X_1$ and $X_2$ spaces we just skip index $i,j$ and
use
$X,x,y,\mu$ etc.

\section{Proof of Theorem~\ref{main}}

We split the proof of Theorem~\ref{main} into a  few  lemmas. First we
show the following straightforward consequences of Assumption~\ref{1b}
  \begin{lemma}\label{ltu}
  Suppose that {\rm (\ref{1.3kernel})} and {\rm (\ref{d})} hold. Then for
all $r,t>0$
    \begin{equation}
    \int_{X-B(y,r)}|p_{t}(x,y)|^2 d\mu(x)
    \le   C    \mu(B(y,t^{1/m}))^{-1}\exp
    \left(-b\frac{r^{{m}/{(m-1)}}}{ {t^{1/(m-1)}}}\right).
    \end{equation}
  In particular
    $$
    \|p_{t}(x,\cdo)\|^2_{L^2(X)}=
    \|p_{t}(\cdo,x)\|^2_{L^2(X)} \le C \mu_j(B(x,t^{1/m}))^{-1}.
    $$
  \end{lemma}
\begin{proof}
By (\ref{1.3kernel}) and  (\ref{d})  (see also \cite[Lemma~2.1]{CD})
\begin{eqnarray*}
\int\limits_{X-B(y,r)}\!\!|p_{t}(x,y)|^2d\mu(x) \le
\frac{C}{\mu(B(y,t^{1/m}))^2}
\!\!\int\limits_{X-B(y,r)}\!\!\exp\left(-2b\sqrt[m-1]{\rho(x,y)^m/t}\right) d\mu(x)\\
\le \frac{C\exp(-b\sqrt[m-1]{r^m/t})}{\mu(B(y,t^{1/m}))^2}
\int_{X}\exp\big(-b\sqrt[m-1]{\rho(x,y)^{m}/t}\big) d\mu(x)\\
\le \frac{C\mu(B(y,t^{1/m}))}{\exp(-b\sqrt[m-1]{r^{m}/t})}.
\end{eqnarray*}
\end{proof}
Second we prove the following lemma

  \begin{lemma}\label{P}
  Suppose that $\|p^{[j]}_{t}(x_j, \cdo)\|^2_{L^2(X_j,\mu_j)} \le C
\mu_j(B(x_j,t^{{1}/{m}}))^{-1}$. Then
    \begin{eqnarray*}
    \Vert  K_{{F}(L_1,L_2)}  (x_P,\cdo)  \Vert^2_{L^2(X_P)}=
    \Vert  K_{\overline{F}(L_1,L_2)}  (\cdo,x_P)  \Vert^2_{L^2(X_P)}
    \nonumber \\ \le C  \Vert F \Vert_{L^\infty}^2 \prod_{j=1,2}\mu_j(B(x_j,R^{-1}))^{-1}
    \end{eqnarray*}
  for any Borel function $F \colon \R^2 \to \C$ such that
  $\supp F\subset [0,{R}^{{m}}]\times [0,{R}^{{m}}]$.
  \end{lemma}
\begin{proof}
Set
$$G_1=\frac{F}{G_2}, \quad \mbox{\rm where}  \quad
G_2(\lambda_1,\lambda_2)=\exp(-{R}^{-{m}}(\lambda_1+\lambda_2)).$$
Then
$$
\Vert G_1(L_1,L_2) \Vert_{L^2(X_P) \to L^2(X_P)} \le
\|G_1\|_{L^\infty}\le  e\|F\|_{L^\infty}.
$$
By (\ref{ker1}) the operator $F(L_1,L_2)$ has  the kernel given by the formula
\begin{equation*}
K_{F(L_1,L_2)}(x_P,y_P)
= \left[{G}_1(L_1,L_2)K_{G_2(L_1,L_2)}(\cdo,y_P)\right](x_P).
\end{equation*}
Now
\begin{eqnarray}
&&\int_{X_P} \vert   K_{F(L_1,L_2)}(x_P,y_P) \vert^{2}  d \mu_P(y_P)\label{GF}
\\ &&\le \Vert
{G}_1(L_1,L_2) \Vert_{L^2 \to L^2}^2  \prod_{j=1,2}\Vert
p^{[j]}_{{R}^{-m}}(\cdo, x_j ) \Vert_{L^2(X_j)}^2\nonumber
\\ &&\le C   \Vert G_1\Vert^2_{L^\infty}
\prod_{j=1,2}\mu_j(B(x_j,{R}^{-1}))^{-1} \le C  \Vert F
\Vert^2_{L^\infty}\prod_{j=1,2}\mu_j(B(x_j,{R}^{-1}))^{-1}  .\nonumber
\end{eqnarray}
\end{proof}
Third we show that (compare \cite{DOS, Ou})
  \begin{lemma}\label{jeden}
  For any $s_1,s_2 \ge 0$ there exists a constant $C$ such that
    \begin{eqnarray}\label{sss}
    \int_{X_1\times X_2} \prod_{j=1,2}\Big(|p^{[j]}_{
(1+i\tau_j){R}^{-m}}(x_j,y_j)|^2\rho_j(x_j,y_j)^{s_j}\Big) d
\mu_1\times\mu_2(x_1,x_2) \nonumber
   \\ \le C  \prod_{j=1,2}\Big(\mu_j(B(y_j,{R}^{-1}))^{-1}
{R}^{-s_j}(1+|\tau_j|)^{s_j}\Big)
    \end{eqnarray}
for all  $\tau_j\in \R$ and ${R}>0$.
  \end{lemma}
\begin{proof}
Note that the integral (\ref{sss}) is a product of integrals and it is
enough to show that the above estimates hold for $j=1$ and $j=2$ separately,
that is
\begin{eqnarray*}
    \int_{X_j}| p^{[j]}_{
(1+i\tau_j){R}^{-{m}}}(x_j,y_j)|^2\rho_j(x_j,y_j)^{s_j} d
\mu_j(x_j)
 \\   \le C   \mu_j(B(y_j,{R}^{-1}))^{-1} {R}^{-s_j}(1+|\tau_j|)^{s_j}
\end{eqnarray*}
holds for $j=1$ and $j=2$. The proof for $j=1$ and $j=2$ is the same so 
to simplify notation we skip the index $j$. The rest of the proof follows closely the proof in
 \cite{DOS} and we describe it here for the sake of completeness. 
First, we assume that $\|f\|_{L^2(X)}=1$ and that $\supp f \subset X-B(y,r)$.
Next, we define the holomorphic function $F_y\colon \{z\in \C \colon\ \Re
e\,z > 0\} \to \C$ by the formula
  $$
  F_y(z)=e^{-zR^m}\mu(B(y,1/R))
  \bigg(\int_X p_z(x,y)f(x) d\mu(x)\bigg)^2.
  $$
By the same argument as in (\ref{GF}) if we put  $z=|z|e^{i\theta}$, then
$\Re e \, z =|z|\cos\theta$ and 
$\|p_{z}(\cdo,y)\|^2_{L^2}=\|p_{|z|\cos\theta}(\cdo,y)\|^2_{L^2}$.
Hence by Lemma~\ref{ltu}
\begin{eqnarray*}
|F_y(z)| & \le &e^{-R^m|z|\cos\theta}\mu(B(y,1/R))\|p_{|z|\cos\theta}(\cdo,y )\|^2_{L^2}
\\ & \le & Ce^{-R^m|z|\cos\theta}
\frac{\mu(B(y,1/R))}{\mu(B(y,\sqrt[m]{|z|\cos\theta}))}
\le
Ce^{-R^m|z|\cos\theta}\Big(1+\frac{1}{R^{m}|z|\cos\theta}\Big)^{d/m}
\\      & \le & CR^{-d}(|z|\cos\theta)^{-d/m}.
\end{eqnarray*}
Similarly for $\theta=0$ by Lemma \ref{ltu}
  $$
  |F_{y}(|z|)|\le C R^{-d} |z|^{-d/m}
  \exp\bigg(-\frac{br^{m/(m-1)}}{|z|^{1/(m-1)}}\bigg).
  $$
Now let us recall the following version of Phragmen-Lindel\"of Theorem
  \begin{lemma}[{\cite[Lemma~9]{Da2}}]\label{dav}
  Suppose that function $F$ is analytic on the half-plane  $\C_+=\{ z\in \C \colon \Re e\, z >
0 \}$
  and that
    $$
    \vert F (|z|e^{i\theta}) \vert \le a_1 (|z|\cos\theta)^{-{\beta_1}}
    $$
    $$
    \vert F(|z|) \vert \le a_1 |z|^{-{\beta_1}}\exp(-a_2|z|^{-{\beta_2}})
    $$
  for some $a_1, a_2 > 0$, ${\beta_1} \ge 0$, ${\beta_2}\in(0,1]$, all
$z \in \C_+$.  Then
    $$
    \vert F(|z|e^{i\theta})  \vert   \le  a_1 2^{\beta_1}
    (|z|\cos\theta)^{-{\beta_1}}
    \exp\Big(-\frac{a_2{\beta_2}}{2}|z|^{-{\beta_2}}\cos\theta\Big)
    $$
  for all $z \in \C_+$.
  \end{lemma}
Now if $|z|e^{i\theta}=(1+i\tau)R^{-m}$, then $|z|=R^{-m}(1+|\tau|^2)^{1/2}$,
$\cos\theta=(1+|\tau|^2)^{-1/2}$ and $|z|\cos\theta  =   R^{-m}$.
Putting   $a_1=CR^{-d}$, $a_2=br^{m/(m-1)}$, ${\beta_1}=d/m$
 and ${\beta_2}=1/(m-1)$ in
Lemma~\ref{dav} we conclude that
\begin{eqnarray*}
|F_y((1+i\tau)R^{-m})| \le C' \exp{\Big(-b'(rR/(1+|\tau|))^{m/(m-1)}\Big)}.
\end{eqnarray*}
Hence
\begin{eqnarray*}
\mu(B(y,1/R))\int_{X-B(y,r)} |p_{(1+i\tau)R^{-m}}(x,y)|^2d \mu(x) \\ \le
 C \exp{\Big(-b'(rR/(1+|\tau|))^{m/(m-1)}\Big)}.
\end{eqnarray*}
Finally, we have
\begin{eqnarray*}
&&\hspace{-2cm}\int_X |p_{(1+i\tau)R^{-m}}(x,y)|^2\ro(x,y)^s  d \mu(x)\\
&=& \sum_{k \ge 0} \int_{ k (1+|\tau|)R^{-1} \le \ro(x,y) \le (k+1)
(1+|\tau|)R^{-1}}
|p_{(1+i\tau)R^{-m}}(x,y)|^2\ro(x,y)^s  d \mu(x)\\
&\le& (1+|\tau|)^s R^{-s}\sum_{k \ge 0} (k+1)^s \int_{ X - B(y,k
(1+|\tau|)R^{-1}) }
|p_{(1+i\tau)R^{-m}}(x,y)|^2  d \mu(x)\\
&\le& C\mu(B(y,1/R))^{-1}R^{-s}(1+|\tau|)^{s}.
 \end{eqnarray*}
\end{proof}
Next we show that
\begin{lemma}\label{waga}
Suppose that $L_1$ and $L_2$ satisfy Assumption~{\rm (\ref{1b})}, $R>0$
and $s=s_1+s_2$. Then for any $\epsilon > 0$ there exists a constant 
$C=C(s,\epsilon)$ such that
\begin{eqnarray*}
\int_{X} \vert K_{F(L_1,L_2)}  ((x_1,x_2),(y_1,y_2))\vert^{2}
\prod_{j=1,2}(1 + {R}\rho_j(x_j,y_j) )^{s_j}
d\mu_1\times \mu_2(x_1,x_2) \nonumber\\
\le C    \Vert \delta_{R^m} F
\Vert_{W^\infty_{\frac{s}{2} + \epsilon}}^{2} \prod_{j=1,2}\mu_j(B(y_j,R^{-1}))^{-1} 
\end{eqnarray*}
for all Borel functions $F$ such that $\supp F \subseteq [0, {R}^{{m}}]^2$.
\end{lemma}
\begin{proof}
Set
$$
G(\lambda_1,\lambda_2)= e^{\lambda_1+\lambda_2}\delta_{{R}^{m}}
F(\lambda_1,\lambda_2).
$$
In virtue of the Fourier inversion formula
\begin{eqnarray*}
&&F(L_1,L_2)=G(R^{-m}L_1, R^{-m}L_2)\exp{(-R^{-m}(L_1+L_2))
}\\&&\hspace{0.5cm}=\frac{1}{4\pi^2}\int_{\R^2}
\exp{\left( \sum_{j=1,2}(i\tau_j-1)R^{-m}L_j \right)}
\widehat{G}(\tau_1,\tau_2)
d \tau_1 d \tau_2,
\end{eqnarray*}
where $\widehat{G}$ is the Fourier transform of  the function $G$. Hence
\begin{eqnarray*}
 &&K_{F(L_1,L_2)}((x_1,x_2),(y_1,y_2))  \\&&\hspace{0.5cm}= \frac{1}{4\pi^2}
\int_{\R^2}\widehat{G}(\tau_1,\tau_2)
\prod\limits_{j=1,2} p^{[j]}_{(1-i\tau_j){R}^{-m}}(x_j,y_j)   d \tau_1 d
\tau_2.
\end{eqnarray*}
Thus   by Lemma~\ref{jeden} and Lemma~\ref{ltu}
  \begin{eqnarray}
    &&\Bigg( \int_{X_1\times X_2} \vert
K_{F(L_1,L_2)}((x_1,x_2),(y_1,y_2))
  \vert^{2}  \nonumber \\ &&\hspace{2cm} \times(1 + R\ro_1(x_1,y_1)
)^{s_1}(1 + R\ro_2(x_2,y_2) )^{s_2}
d\mu_1\times \mu_2(x_1,x_2) \Bigg)^{\frac{1}{2}}   \nonumber\\
  &&\le\int\limits_{\R^2} \! \vert  \widehat{G}(\tau_1,\tau_2)\vert
  \Big(\prod_{j=1,2}\int\limits_{X_j} \! |p^{[j]}_{(1-i  \tau_j)R^{-m}}(x_j,y_j)|^2
(1+R\ro_j(x_j,y_j))^{s_j}
  d\mu_j(x_j)\Big)^{\frac{1}{2}} d  \tau_1 d \tau_2  \nonumber\\
  &&\le C\Big( \prod_{j=1,2}\mu_j(B(y_j,R^{-m})^{-\frac{1}{2}}\Big)
    \int_{\R^2}
  \vert \widehat{G}(\tau_1,\tau_2) \vert
(1+|\tau_1|)^{s_1/2}(1+|\tau_2|)^{s_2/2}
  d \tau_1 d \tau_2  \nonumber\\
  &&\hspace{1cm} \le  C  \Big(\prod_{j=1,2}\mu_j(B(y_j,R^{-m})^{-\frac{1}{2}}\Big)\left(\int_{\R^2} \vert \widehat{G}(\tau_1,\tau_2) \vert^{2}
  ({1+\tau_1^2+\tau_2^{2}})^{\frac{s_1+s_2+\epsilon+2}{2}}\right)^{\frac{1}{2}}
\nonumber \\ &&\hspace{2cm}
 \times 
  \left(\int_{\R^2}({1+\tau_1^2+\tau_2^{2}})^{\frac{-2-\epsilon}{2}}\right)^{\frac{1}{2}}
\label{dwa}\\
  &&  \hspace{4cm} \le C   
  \|G\|_{W^2_{\frac{s_1+s_2+2+\epsilon}{2}}}\prod_{j=1,2}\mu_j(B(y_j,R^{-{m}})^{-\frac{1}{2}}.
\nonumber
\end{eqnarray}
However, $\supp F  \subseteq  [0, {R}^{{m}}]^2$ and $\supp \delta_{{R}^m}F
\subseteq [0,1]^2$ so
\begin{equation}\label{dwa2}
  \|G\|_{W^2_{{(s+2+\epsilon)}/{2}}}
\le C \|\delta_{R^{m}} F\|_{W^2_{(s+2+\epsilon)/{2}}}
\le C \|\delta_{R^{m}} F\|_{W^\infty_{(s+2+\epsilon)/{2}}}.
\end{equation}
 From the last two estimates one can obtain a multiplier result in which
the required order of differentiability of
the function $ F$ is of $1$ greater than that of Lemma~\ref{waga}.
To get rid of this additional $1$ we use Mauceri-Meda interpolation argument, see
\cite{MM} and \cite{DOS}.
First we note that the estimates from Lemma~\ref{waga}  are  equivalent to the
following inequality
\begin{eqnarray}
\int_{X} \vert K_{\delta_{R^{-m}}F(L_1,L_2)}
((x_1,x_2),(y_1,y_2))\vert^{2} \prod_{j=1,2}(1 +
R\rho_j(x_j,y_j) )^{s_j}
d\mu_1\times \mu_2(x_1,x_2)
\nonumber
\\ \le C  
\Vert F\Vert_{W^\infty_{ {\frac{s_1+s_2}{2}}+ \epsilon}}^{2} \prod_{j=1,2}  \mu_j(B(y_j,R^{-1}))^{-1} \label{dww}
\end{eqnarray}
for all bounded Borel functions $F$ such that $\supp F\subset [0,1]^2$.
Now we  define the linear operator $K_{y_1,y_2,{R}} \colon
L^\infty([0,1]^2) \to
L^2{(X_1\times X_2,\mu_1\times \mu_2)}$ by the formula
$$
K_{y_1,y_2,{R}}(F)=K_{\delta_{R^{-m}} F(L_1,L_2)}(\cdo,(y_1,y_2)).
$$
By Lemma~\ref{P}
$$
\Big\Vert K_{y_1,y_2,{R}}\Big\Vert_{L^\infty([0,1]^2) \to L^2(X_1 \times
X_2 ,\mu_1 \times \mu_2)}^2\le C  \prod_{j=1,2} 
\mu_j(B(y_j,R^{-1}))^{-1}.
$$
Next we put  $L^2_{y_1,y_2,s_1,s_2,{R}}=L^2(X_1\times
X_2,\mu_{y_1,y_2,s_1,s_2,{R}})$, where
$$
d\mu_{y_1,y_2,s_1,s_2,R}(x_1,x_2)
=\prod_{j=1,2}(1+R\rho_j(x_j,y_j))^{s_j}d\mu_1\times \mu_2(x_1,x_2)
$$
By (\ref{dwa}) and (\ref{dwa2})
$$
 \Big\Vert K_{y,R}\Big\Vert_{W^\infty_{(s_1+s_2+2+\epsilon)/2}([0,1]^2)
\to L^2_{y_1,y_2,s_1,s_2,R}}^2
\le C    \prod_{j=1,2}  \mu_j(B(y_j,R^{-1}))^{-1}.
$$
By interpolation, 
for every $\theta \in (0,1)$ there  exists a constant $C$ such
that
\begin{eqnarray*}
&& \Vert \delta_{R}F(L_1,L_2)(\cdo,(y_1,y_2) )
\Vert_{L^{2}_{y_1,y_2,\theta s_1, \theta s_2,R}} ^2\\&& \le  {C}\Vert F \Vert_{{[L^\infty, W^\infty_{(s_1+s_2+2+\epsilon)/2}]_{[\theta]}}}^2 \prod_{j=1,2}
\mu_j(B(y_j,R^{-1}))^{-1}.
\end{eqnarray*}
In particular, for all $s>0$, $ \theta \in (0, 1)$ and $\epsilon' >
\epsilon $
\begin{eqnarray*}
&&\Vert \delta_{R^{-m}}F(L_1,L_2)(\cdo,(y_1,y_2) )
\Vert_{L^{2}_{y_1,y_2,\theta s_1, \theta s_2,R}}^2 \\&& \le  {C}\Vert F \Vert_{{ W^\infty_{(\theta s_1+\theta s_2+2\theta +\theta
\epsilon')/2}}}^2 \prod_{j=1,2}
\mu_j(B(y_j,R^{-1}))^{-1}.
\end{eqnarray*}
Hence putting $s'_j={s_j}/{\theta}$ in  the above  inequality and  taking
$\theta$ small enough we obtain
\begin{eqnarray*}
&&\Vert \delta_{R^{-m}}F(L_1,L_2)(\cdo,(y_1,y_2) )
\Vert_{L^{2}_{y_1,y_2,s'_1, s'_2,R}}^2 \\ &&\le  {C} \Vert F \Vert_{{ W^\infty_{(s'_1+s'_2)/2+\epsilon"}}}^2\prod_{j=1,2}
\mu_j(B(y_j,R^{-1}))^{-1}.
\end{eqnarray*}
 This proves (\ref{dww}) and Lemma~\ref{P}
\end{proof}
Next we show the following simple consequence of doubling condition (\ref{d}).
\begin{lemma}\label{waga2}
Suppose that {\rm (\ref{d})} holds. Then for any $s_1 > d_1$  and $s_2 >
d_2$ there exists $\epsilon>0$ such that
\begin{equation}
\int_{{X_P-B(y_P,r)}}\prod_{j=1,2}(1 + {{R}}\rho_j(x_j,y_j) )^{-s_j}
 d\mu(x) \le C  {(1+rR)^{-\epsilon}}\prod_{j=1,2}
\mu(B(y_j,{R}^{-1}))\end{equation}
where $B(y_P,r)=\{x_P \in X_P \colon \, \rho_P(x_P,y_P) < r\}$.
\end{lemma}
\begin{proof}
Choose $\epsilon >0$ such that $s'_j= s_j-\epsilon >d_j$ and note that
$$
\inf_{{X_P-B(y_P,r)}}\prod_{j=1,2}(1 + {R}\rho_j(x_j,y_j) )^{-\epsilon}\le
{(1+rR)^{-\epsilon}}.
$$
Hence
\begin{eqnarray*}
\int_{{X_P-B(y_P,r)}}\prod_{j=1,2}(1 + {R}\rho_j(x_j,y_j) )^{-s_j}
 d\mu_P(x_P)
\\ \le {(1+rR)^{-\epsilon}} \int_{{X_P}}     \prod_{j=1,2}(1 +
{R}\rho_j(x_j,y_j) )^{-s'_j}
d\mu_P(x_P).
\end{eqnarray*}
Next
$$
\int_{{X_P}}     \prod_{j=1,2}(1 + {R}\rho_j(x_j,y_j) )^{-s'_j}
d\mu_P(x_P) =\prod_{j=1,2} \int_{{X_j}}    (1 + {R}\rho_j(x_j,y_j) )^{-s'_j}
d\mu_j(x_j)
$$
so it is enough to show that for $j=1$ and $j=2$ one has
$$
\int_{{X_j}}  (1 + {R}\rho_j(x_j,y_j) )^{-s'_j}
d\mu_j(x_j) \le \mu_j(B(y_j,{R}^{-1})).
$$
As before we skip the index  $j$ in the proof.
\begin{eqnarray*}\label{rr1}
\int\limits_{X} (1+ R\ro(x,y)  )^{-s}d\mu(x) \le  \mu(B(y, R^{-1}))  +  \sum_{k \ge 0}
\int\limits_{2^k \le R\ro(x,y)  \le 2^{k+1}} (R  \ro(x,y))^{-s}d\mu(x)
 \\
\le  \mu(B(y, R^{-1})) +\sum_{k \ge 0} (2^k)^{-s}  \mu(B(y, 2^{k+1}/R)) \le
C\sum_{k \ge
0} (2^k)^{d-s} \mu(B(y, R^{-1}))\\
\le  \mu(B(y, R^{-1})).
\end{eqnarray*}
\end{proof}

To prove that operator is of weak type (1,1) we usually use estimates for
the gradient  of the  kernel.  The following theorem replaces the gradient
estimates in the proof of Theorem~\ref{main}.

\begin{theorem}\label{CZ}
Suppose that $\| F \|_{L^\infty} \le C_1$ and that
\begin{equation}\label{lap1}
\sup_{r\in\R^+}\sup_{y_P\in X_P}\int_{ X_P-B(y_P,r)}
|K_{F(1-\Phi_{r})(L_1,L_2)}(x_P,y_P)| d \mu_P(x_P) \leq C_1 ,
\end{equation}
where $\Phi_{r}(L_1,L_2)=\exp(-r^m(L_1+L_2) )$.  Then
$$
 \| F(L_1,L_2)\|_{L^1\to L^{1,\infty}} \le C C_1.
$$
\end{theorem}
\begin{proof}
Theorem~\ref{CZ} follows from  \cite[Theorem~2]{DM}. Indeed one can
check that the family of operators
$\exp(-tL_P)$, where $L_P=L_1+L_2$ satisfies all assumptions from
\cite{DM}. Hence any operator $T$ which is bounded on $L^2$ and whose kernel
$K_{T(1-\Phi_{r})}(L_1,L_2)$ satisfies condition (\ref{lap1}) is of weak type
$(1,1)$. See also \cite{CD, He1} and \cite{CoS} for similar results.
\end{proof}
{\em Remark}. In this note we define Calder\'on-Zygmund type operators in the sense of Theorem~\ref{CZ}
and  \cite[Theorem~2]{DM}. Note that this does not mean that the singular kernel of the operator $T$ is 
continuous outside the diagonal as it is the case  for classical Calder\'on-Zygmund operators. For quasielliptic operators which we discuss below  continuity of the 
kernel is a question of independent interest. 

\begin{proof}[Proof of Theorem~\ref{main}]
Choose   a function~$\eta$  in~$C_c^\infty(0,1)$  such that
\begin{equation*}
\sum_{n\in \Z} \eta(2^{nm}\lambda) = 1 \qquad\forall \lambda \in \R_+.
\end{equation*}
  Then
  \begin{equation*}
  F(1-\Phi_{r})(L_1,L_2) =\sum_{n\in \Z} \wet_n F(1-\Phi_{r})(L_1,L_2),
  \end{equation*}
where $\wet_n=\delta_{2^{-nm}}\wet$. Recall that $\wet(\lambda_1,\lambda_2)=\eta(\lambda_1+\lambda_2)$. 
By Lemma~\ref{waga} and Lemma~\ref{waga2} for any $s_j>{d_j}$, $j=1,2$
there exists $\epsilon, \epsilon' >0$ such that
\begin{eqnarray*}
&&
\int_{X_P-B(y_P,r)} | K_{\wet_nF(1-\Phi_{r})(L_1,L_2)}(x_P,y_P)|{d}
\mu_P(x_P)\\  &\le&  
\Big(\int_{X_P}
\vert K_{\wet_n F(1-\Phi_{r})(L_1,L_2)}
(x_P,y_P)\vert^{2}\prod_{j=1,2} (1 +   2^n\ro(x_j,y_j) )^{s_j}
d\mu_P(x_P)\Big)^{1/2}\nonumber  \\
&&\times\Big(\int_{{X_P-B(y_P,r)}}\prod_{j=1,2} (1 +   2^n\ro(x_j,y_j)
)^{-s_j}
d\mu_P(x_P)\Big)^{1/2}    \nonumber  \\
 &\leq&  C
(1+2^{n}r)^{-\epsilon/2} \Big( \prod_{j=1,2}\mu_j(B(y_j,{2}^{-n}))^{1/2}\Big) \nonumber  \\
&&\times \Big( \prod_{j=1,2}\mu_j(B(y_j,{2}^{-n}))^{-1/2}  \Big)  \|\delta_{2^{nm}}   [\wet_n
F(1-\Phi_{r})]\|_{W^\infty_{\frac{s_1+s_2+\epsilon'}{2}}}\nonumber
\\
 &\leq&  C
 (1+2^{n}r)^{-\epsilon/2}\|\delta_{2^{nm}}[\wet_n
F(1-\Phi_{r})]\|_{W^\infty_{\frac{s_1+s_2+\epsilon'}{2}}}.\nonumber
\end{eqnarray*}
Now for any Sobolev space $W_s^{p}(\R)$ and any integer $k>s$ 
\begin{eqnarray*}
\|\delta_{2^{nm}}[\wet_n   F(1-\Phi_{r})]\|_{W^{\infty}_{s}}   \leq   C
\|\delta_{2^{nm}}[\wet_n F]\|_{W_s^{p}}
\|\delta_{2^{nm}}[1-\Phi_{r}]\|_{C_k([1/4,1])}\\ \le
\frac{C 2^{nm}r^m}{1+2^{nm}r^m} \|\delta_{2^{nm}} [\wet_n F]\|_{W_{s}^{\infty}}
=\frac{C (r2^n)^{m}}{1+(r2^n)^{m}} \|\wet\delta_{2^{nm}}  F\|_{W_s^{\infty}}.
\end{eqnarray*}
Hence for $s=\frac{s_1+s_2+\epsilon'}{2}$
\begin{eqnarray}\label{koniec}
&&\hspace{-1cm}\sup_{y_P\in X_P}  \int_{X_P-B(y_P,r)}
|K_{F(1-\Phi_{r})(L_1,L_2)}(x_P,y_P)| d \mu_P(x_P)   \\
&\leq& C \sum_{n\in \Z} \frac{(r2^n)^{m}}{1+(r2^n)^{m}} \, (1+2^{n}r)^{-\epsilon/2}
 \|\wet\delta_{2^{nm}}  F\|_{W_s^{p}}
\leq C\sup_{t>0}  \|\wet\delta_{t}  F\|_{W_s^{p}} \nonumber
\end{eqnarray}
as required to prove  Theorem~\ref{main}.
\end{proof}

\section{Analysis on fractals and quasielliptic operators}\label{chfra}

Many interesting examples of spaces and operators satisfying
Assumption~\ref{1b} are described in the theory
of Brownian Motion  and Laplace like operators on fractals, see for example
\cite{Ki, Str2}.
A compelling  instance of such an operator  is  the Laplace operator on the Sierpi\'nski gasket $SG$
(Neumann or Dirichlet).
Assumption~\ref{1b} and condition (\ref{d}) hold with $d =\log 3/ (\log 5
-\log 3)=$ and
$m=d+1=\log5/ (\log5-\log3) $, see \cite{BP, Str1, Str2}. We are especially interested in this example because   
the set of ratios of eigenvalues for the Laplacian on $SG$ has gaps so it can be used to construct quasielliptic operators. Indeed the main results of \cite[Theorem~2.1]{BS}
says that  
\begin{theorem} For any two eigenvalues $\lambda_i, \lambda_j$ of the Laplacian on SG
$$
\frac{\lambda_i}{\lambda_j}\notin (\alpha,\beta)
$$
 where $\alpha$ and $\beta$ are defined by {\rm (\ref{gap})}.
\end{theorem}
 
 There are more examples of Laplacian on fractals such that the sets of ratios of their eigenvalues
 have gaps. For instance  this is the case for  hierarchical fractals introduced by Humbly $HH(b)$. For definitions 
 and a precise formulation of results, see \cite[Theorems~5.2~and~5.3]{DS}.  It  also can be noticed that the discussed result that holds on $SG$ also
 holds in the setting of \cite{Te}.  It is easy to note that if $(\alpha,\beta)$ is a gap in the set of ratios of eigenvalues $\alpha<\gamma < \beta$, $a,b\in \R$, $c,d>0$ and $\frac{d}{c}=\gamma$ then the operators defined by (\ref{qell}) are bounded on $L^2$.
 It seems likely that this short description of results concerning quasielliptic operators and gaps in 
 the sets of ratios of eigenvalues will be soon outdated. Therefore, we state our result   in an abstract way
 which could be applied to all quasielliptic operators. 
 
 \begin{definition}\label{one}
 Suppose that $L$ is a positive self-adjoint operator acting on $L^2(X)$. We say that the interval
$(\alpha, \beta)$ is a gap in the set of ratios of eigenvalues of $L$ if
$$
\frac{\lambda_i}{\lambda_j}\notin (\alpha,\beta)
$$
for all $\lambda_i,\lambda_j$ in the $L^2$ spectrum of the operator $L$. 
\end{definition}
 
Next we consider the Cartesian product of two copies of metric measure space $X$ and
operators $L_{1}=L \otimes 1$ and $ L_2=1 \otimes L$ acting on $L^2(X^2)$.

 \begin{theorem}\label{frac}
Suppose that metric measure space $(X,\mu, \rho)$ satisfies condition {\rm (\ref{d})} with doubling constant $d$
and that operator $L$ satisfy Assumption~{\rm\ref{1b}}.  Next assume that 
$(\alpha, \beta)$ is a gap in the set of ratios of eigenvalues of $L$ and that
$$
\mbox{\rm supp}\, \omega \subset \{(\lambda_1,\lambda_2)\colon
\, \left|\frac{\lambda_1}{\lambda_2}-\gamma \right| \le \sigma\}
$$
for some $\gamma, \sigma$ such that $\alpha < \gamma-\sigma<\gamma+\sigma < \beta$. 
Assume in addition that
for some $s>{d}$
$$
\sup_{t> 0} \Vert \widetilde{\eta} \, \delta_{t }\left[(1-\omega) F\right]
\Vert_{W^\infty_s}<\infty,
$$
where $ \widetilde{\eta}$ is defined in the same way as in
Theorem~{\rm\ref{main}}. 
Then  the operator $F(L_1,L_2)$ is of weak type $(1,1)$ and is bounded on
$L^q(X^2)$
  for all $1<q<\infty$.
\end{theorem}
\begin{proof}
It is not difficult to note that if $(\alpha, \beta)$ is a gap in the set of ratios of eigenvalues of $L$ then
  $\omega F(L_1,L_2)=0$ so $F(L_1,L_2)=(1-\omega) F(L_1,L_2)$ and
Theorem~\ref{frac} follows from
Theorem~\ref{main}.
\end{proof}
 Below we describe a straightforward consequence of Theorem~\ref{frac}.
\begin{coro}\label{ko}
Suppose that for some $c,d>0$, $cL_1-dL_2$ is a quasielliptic operator, that is $\alpha< \frac{d}{c}< \beta$, 
where $(\alpha, \beta)$ is a gap in the set of ratios of eigenvalues of $L$ described in Definition~{\rm \ref{one}}.
Then for all constants $a,b \in \C$ the operators
$$
\frac{aL_1+bL_2}{cL_1-dL_2} \quad \mbox{and} \quad
\frac{L_1L_2}{(cL_1-dL_2)^2}
$$
are bounded on all $L^p$ spaces for $1<p<\infty$ and of a weak type $(1,1)$.
\end{coro}
\begin{proof}
Corollary~\ref{ko} follows from Theorem~\ref{frac} if one chooses function $\omega$ such that 
$\omega(\lambda_1,\lambda_2)=1$ for all  $\lambda_1,\lambda_2$ such that $\left|\frac{\lambda_1}{\lambda_2}-\frac{d}{c}\right| \le \epsilon $ for some small $\epsilon >0$. 
\end{proof}

\begin{ex}
If $\alpha$ and $\beta$ are defined by {\rm (\ref{gap})} then the Neumann and Dirichlet Laplacian acting on 
Sierpi\'nski gasket $GS$ satisfy all assumption of Theorem~{\rm \ref{frac}} and Corollary~{\rm \ref{ko}}.
\end{ex}

\bigskip

{\bf Acknowledgement: } The author is indebted to Robert Strichartz for alerting him to  the main research problem studied in the paper and  hosting him at Cornell University.

\end{document}